\documentclass[reqno, 12pt]{amsart}
\usepackage{graphicx} 
\usepackage[utf8]{inputenc}
\usepackage[british]{babel}
\usepackage[T1]{fontenc}
\usepackage{lmodern}
\usepackage{amsmath,amssymb,amsthm,todonotes,booktabs,xcolor,graphicx,longtable, bbm, nicefrac, xfrac, adjustbox}
\usepackage{url, hyperref}
\usepackage{appendix}
\usepackage{tikz-cd}
\usepackage[a4paper]{geometry}

\definecolor{dblue}{rgb}{0,0,0.7}
\definecolor{dartmouthgreen}{rgb}{0.05, 0.5, 0.06}
\definecolor{dlilac}{rgb}{0.6, 0.33, 0.73}
\definecolor{ffuchsia}{rgb}{0.96, 0.0, 0.63}
\definecolor{hpurple}{rgb}{0.32, 0.09, 0.98}
\newtheorem{thm}{Theorem}
\newtheorem{prop}[thm]{Proposition}
\newtheorem{cor}[thm]{Corollary}
\newtheorem{lem}[thm]{Lemma}

\newtheorem{rem}[thm]{Remark}
\numberwithin{thm}{section}
\numberwithin{equation}{section}
\theoremstyle{definition}
\newtheorem{defi}{Definition}
\newtheorem{exa}{Example}
\numberwithin{defi}{section}
\numberwithin{exa}{section}

\newcommand{\FF}{\mathbb{F}}
\newcommand{\QQ}{\mathbb{Q}}
\newcommand{\ZZ}{\mathbb{Z}}
\newcommand{\RR}{\mathbb{R}}
\newcommand{\CC}{\mathbb{C}}

\newcommand{\NN}{\mathbb{N}}
\newcommand{\AAA}{\mathbb{A}}
\newcommand{\PR}{\mathfrak{p}}
\newcommand{\cc}{\mathfrak{c}}

\DeclareMathOperator{\im}{Im} 
\DeclareMathOperator{\re}{Re} 

\DeclareMathOperator{\Res}{Res}
\DeclareMathOperator{\fin}{fin}
\DeclareMathOperator{\Hom}{Hom}
\DeclareMathOperator{\Gal}{Gal}
\DeclareMathOperator{\Frob}{Frob}
\DeclareMathOperator{\ext}{-ext}
\DeclareMathOperator{\br}{Br} 
\DeclareMathOperator{\ch}{H}
\DeclareMathOperator{\id}{id}
\title{Asymptotic for Galois group isomorphic to C3}
\date{November 2023}
\vspace{10cm}
\vspace{50cm}
\vfill
\newpage

\title{Abelian number fields with frobenian conditions}
\author{Julie Tavernier}
\date{July 2024}
\address{Julie Tavernier, Department of Mathematical Sciences, University of Bath,
Claverton Down, Bath, BA2 7AY, UK}
\email{jlt86@bath.ac.uk}
\begin{document}
\begin{abstract}
    We study the distribution of abelian number fields with frobenian conditions imposed on the conductor. In particular we find an asymptotic for the number of abelian field extensions of a number field $k$ whose conductor is the sum of two squares. We also discuss an application of the Brauer group of stacks to quadratic number fields.
\end{abstract}
\maketitle

\tableofcontents

\section{Introduction}

\subsection{Counting number fields}

Much work has been done on the theory of counting number fields. A theorem by Hermite and Minkowski states that for any integer $B >0$, there are only finitely many number fields $K$ satisfying $|\Delta_{K}| \leq B$, where $\Delta_{K}$ is the discriminant of $K$. A natural question that arises is whether it is possible to count these number fields. In \cite{malle1} and \cite{malle2}, Malle proposed a conjecture giving an asymptotic for the number of number fields of bounded discriminant and given Galois group. Indeed, he asserted that the asymptotic should be of the form \[c_{k,G,B} B^{a(G)}(\log B)^{b(k,G)-1}\] for some explicit $a(G)$ and $b(k,G)$. While this has not been proven in general, work has been done on specific cases where the number fields satisfy certain properties. For example, Wright \cite{Wright1989DistributionOD} found an asymptotic for the number of abelian number fields of bounded discriminant. In this paper we find a general counting result for abelian field extensions whose conductor satisfies some frobenian condition, using the theory of frobenian multiplicative functions. A special case of our result is the following.

\begin{thm} \label{SOTS}

Let $G$ be a finite abelian group, $k$ a number field and $\Phi(E/k)$ the norm of the conductor of the field extension $E/k$. Let $N(k,G, B)$ be the counting function \[N(k,G, B) = \# \{E/k : \Gal(E/k) \cong G, \Phi(E/k) = a^2+b^2, \Phi(E/k) \leq B\}\] that counts the abelian extensions of $k$ whose conductor is the sum of two squares. Then $N(k,G, B)$ satisfies the following asymptotic formula \[N(k,G,B)  \sim c_{k,  G}B(\log B)^{\varpi(k,G, f_{\square})-1} \quad \text{as} \; B \rightarrow \infty\] for a constant $c_{k,  G}$ that is positive and exponent $\varpi(k,G, f_{\square})$ given by \[\varpi(k,G, f_{\square}) = \frac{1}{2}\sum_{\substack{g \in G \backslash\{id\} \\ \mu_4 \not\subset k(\zeta_{|g|})}}\frac{1}{[k(\zeta_{|g|}):k]} + \sum_{\substack{g \in G \backslash\{id\} \\ \mu_4 \subset k(\zeta_{|g|})}}\frac{1}{[k(\zeta_{|g|}):k]}\]
    
\end{thm}

An interesting feature of the proof is that there is no need to assume the existence of such an extension to prove that the leading constant is positive, and thus from this we deduce the following corollary.

\begin{cor} \label{cor 1}
    Let $G$ be a finite abelian group. Then there exists $E/k$ with $\Gal(E/k) \cong G$ such that $\Phi(E/k)$ is the sum of two squares.
\end{cor} It is also possible to prove this result using a constructive argument, by constructing cyclotomic extensions of $\QQ$ of the form $\QQ(\zeta_p)$ for a prime $p \equiv 1 \bmod 4$. A more general version of this result, concerning the splitting of primes in finite extensions can also be obtained using the methods in this paper, as we now explain. 

\subsection{Frobenian multiplicative functions}

The main result in this paper concerns frobenian multiplicative functions, a class of multiplicative functions linked to the theory of frobenian functions described by Serre in \cite[$\S 3.3-3.4$]{serre2016lectures} and which are further explored by Loughran and Matthiesen in \cite[$\S 2$]{Loughran2019FrobenianMF} and Santens in \cite[$ \S 4.1$]{Santens_2023}. These functions encode the splitting behaviour of places of a number field $k$, and examples include the divisor function, the indicator function of the sum of two squares, and the indicator function for integers whose prime factorisation only contains primes $p \equiv a \bmod q$ for fixed integers $a$ and $q$. We recall the full definition in $\S 2$. Let $G$ be a non-trivial finite abelian group and $\mathfrak{c}(E/k)$ denote the conductor (as an ideal) of a field extension $E/k$. For a fixed frobenian multiplicative function $f$ we consider the mean value \begin{equation} \label{counting}
    \sum_{\substack{E/k \\\Phi(E/k) \leq B \\ \Gal(E/k) \cong G}}f(\mathfrak{c}(E/k)).
\end{equation} Then the main result of this paper is the following.\begin{thm}\label{short main thm}
   Let $G$ be a non-trivial finite abelian group and $f$ be a frobenian multiplicative function. Then the counting function (\ref{counting}) satisfies an asymptotic of the form \[\sum_{\substack{E/k \\\Phi(E/k) \leq B \\ \Gal(E/k) \cong G}}f(\mathfrak{c}(E/k)) = c'_{k,G,f}B(\log B)^{\varpi(k,G, f)-1}(1+o(1))\] where the exponent $\varpi(k,G, f)$ defined in $\S$\ref{exp section} depends on $f$ and $c_{k,G,f}$ is a non-negative constant, which is determined explicitly in $\S$\ref{calc leading const}.
\end{thm} It is clear that Theorem $1.1$ is a special case of this theorem, obtained by setting the frobenian multiplicative function $f$ in the theorem to be the indicator function of the sums of two squares. We will prove this result by using a Poisson summation formula derived using harmonic analysis, a method developed by Frei, Loughran and Newton in \cite{HNPabelianext} and \cite{Frei_2022}. It is possible to deduce a version of Theorem \ref{SOTS}
from the work of Alberts and O'Dorney \cite{brandonevan} (with its Corrigendum \cite{corrigendum}) and that of Alberts \cite{Alberts2021HarmonicAA}, however their result does not given an explicit expression for the leading constant. We are able to obtain this formula in Section \ref{calc leading const}, and such an explicit formula is important in applications. 

Contrary to Theorem $1.1$, there is no reason why the leading constant in Theorem $1.3$ should be positive in general. Positivity can fail for multiple reasons. For example, the set of all cyclic cubic extensions of $\QQ$ such that every prime $p$ ramifying in the extension satisfies $ p \equiv 2 \bmod 3$ is empty. This can be shown by considering the frobenian multiplicative function given by the indicator function for those numbers whose prime factors are all congruent to $2 \bmod 3$, and showing that the leading constant is zero in this case. This is an example of a frobenian behaviour preventing the existence of such extensions. There is also the famous counter-example of Grunwald and Wang, which asserts that there are no $\ZZ/8\ZZ$ extensions of $\QQ$ which are totally inert at two, which shows that pathological behaviour at a single prime can also obstruct the existence of certain field extensions. These examples also tell us that it is not in general possible to find a simple expression for the exponent $\varpi(k,G,f)$ in the asymptotic formula.

\subsection{The classifying stack $BG$}

Over quadratic number fields, it is possible to interpret the condition of the discriminant being the sum of two squares via the specialisation of elements in the Brauer group of the classifying stack $BG$. This is related to Serre's work on the specialisation of Brauer group elements in \cite{Serre2000SpcialisationD}, but working over stacks rather than varieties. In the case of a quadratic field this recovers the result on the conductor, as the discriminant of a quadratic field extension is equal to its conductor. This will be discussed in more detail in $\S 5$.

\subsection{Methodology and structure}

We begin by introducing properties of frob\-enian multiplicative functions as described by Loughran and Matthiesen in \cite[$\S 4.1$]{Loughran2019FrobenianMF} and state a result on the zeta function of a frobenian function that we will require later in the paper. In $\S 3$ we prove the main result about the asymptotic for abelian number fields with frobenian conditions imposed on the conductor by using the harmonic analysis techniques introduced by Frei, Loughran and Newton in \cite[$\S 3.4$]{Frei_2022}. Moreover, we will find an explicit expression for the leading constant in terms of sums of Euler products and use this to analyse when it can be positive. In $\S 4$ we prove Theorem \ref{SOTS} and its corollary by applying the main result from $\S 3$ to the indicator function of the sums of two squares, as well as proving the explicit formula for the exponent $\varpi(k,G,f_{\square})$ in this case. We also generalise this result to counting completely split primes in a field extension $L/k$. Finally in $\S 5$, we motivate some of the theory about counting conductors that are the sums of two squares via the theory of stacks. We show that statements on the discriminant of number fields can be recovered by proving results on the specialisation of certain Brauer group elements in the Brauer groups of the stack $BG$ for a finite constant group scheme $G$.

\subsection{Notation}

\begin{itemize}
    \item $k$ is a number field
    \item $\Omega_k$ the set of places of $k$
    \item $v$ a place of $k$ 
    \item $\mathcal{O}_k$ the ring of integers of $k$
    \item $k_v$ the completion of $k$ at $v$ 
    \item $\mathcal{O}_v$ the ring of integers of $k_v$. When $v|\infty$, we use the convention $\mathcal{O}_v = k_v $
    \item $q_v$ is the cardinality of the residue field
    \item $\AAA^{\times}$ are the ideles of $k$
    \item $I_k$ the fractional ideals of $k$
    \item $\mathfrak{p}_v$ maximal ideal of $\mathcal{O}_v$
    \item $\mathfrak{c}(E/k)$ is the conductor as an ideal of $E/k$
    \item $\Phi(E/k)$ is the norm of the conductor $\mathfrak{c}(E/k)$
\end{itemize}
If $v$ is a finite place of $k$, we write $\mathfrak{p}_v$ for the corresponding prime ideal. We will say that an ideal $ \mathfrak{a} \subset \mathcal{O}_k$ is the sum of two squares if its norm $N(\mathfrak{a})$ is the sum of two squares. In particular, we say the conductor $\mathfrak{c}(E/k)$ of a number field $k$ is the sum of two squares if its norm $\Phi(E/k)$ is the sum of two squares. Furthermore, we will make use of the following definition.

\begin{defi}[$G$-extensions of $k$]
    Let $G$ be a finite abelian group, $F$ be a field and fix a choice of separable closure $\overline{F}$ of $F$. A \textit{$G$-extension} of $F$ is a surjective continuous homomorphism $\Gal(\overline{F}/F) \rightarrow G$. A \textit{sub-$G$-extension} of $F$ is a continuous homomorphism $\Gal(\overline{F}/F) \rightarrow G$. In other words, it is a pair $(L/F, \psi)$, where $\psi: \Gal(L/F) \rightarrow G$ is an injective homomorphism and $L/F$ is some field extension of $F$ contained inside $\overline{F}$.
\end{defi}

Choosing a $G$-extension of $k$ is equivalent up to multiplication by the factor $|\text{Aut}(G)|$ to choosing a field extension $k \subset E \subset \overline{k}$ and an isomorphism $\Gal(E/k) \xrightarrow{\sim} G$. Forgetting the choice of isomorphism just means that the counting results need to be multiplied by $|\text{Aut}(G)|$.  Let $G\ext(k)$ denote the set of all $G$-extensions of $k$, and if $\varphi$ is a $G$-extension of $k$, then $E_{\varphi}$ will denote the corresponding number field, $\mathfrak{c}(\varphi)$ the conductor of the corresponding number field and $\Phi(\varphi)$ the norm of the conductor. For every place $v$ of $k$ we fix a separable closure $\overline{k}_v$, and embeddings $k \hookrightarrow \overline{k} \hookrightarrow \overline{k}_v$ and $k \hookrightarrow k_v \hookrightarrow \overline{k}_v$ that are compatible with each other. In this way a sub-$G$-extension $\varphi$ of $k$ induces a sub-$G$-extension $\varphi_v$ of $k_v$ at every place of $k$. Thus we have reduced the problem of counting number fields to one of counting surjective homomorphisms $\varphi: \Gal(\overline{k}/k) \rightarrow G$. From now on we will use the viewpoint of $G$-extensions of $k$ rather than that of abelian field extensions $E/k$, and we define the following counting function  \[N(k,G,B,f) = \sum_{\substack{\Phi(\varphi) \leq B \\ \varphi \in G-\text{ext}(k)}}f(\mathfrak{c}(\varphi)).\]

\subsection{Acknowledgements}

I would like to thank Daniel Loughran for his guidance and many useful discussions, as well as feedback on drafts of this paper. I would also like to thank Brandon Alberts for his helpful comments. I am supported by the EPRSC Postdoctoral Studenship.

\section{Frobenian multiplicative functions}

We use the theory of frobenian multiplicative functions. If $k$ is a number field, we will write $\mathfrak{p}$ for the primes of $k$, which are prime ideals of $\mathcal{O}_k$. We use the following definition for frobenian multiplicative functions from \cite[Def. $4.1$]{Santens_2023}.

\begin{defi}[Frobenian multiplicative functions]\label{frob funct}
    Let $S$ be a finite set of places of a number field $k$. A multiplicative function $f: I_k \rightarrow \CC$ is an \emph{$S$-frobenian multiplicative function} if it satisfies the following:
    \begin{itemize}
        \item There exists some constant $H \in \NN$ such that $|f(\mathfrak{p}^n)| \leq H^n$ for all primes $\mathfrak{p}$ and $n \in \NN$.
        \item For all $\epsilon > 0 $ there exists a constant $C_{\epsilon}$ such that $|f(\mathfrak{n})| \leq C_{\epsilon}N(\mathfrak{n})^{\epsilon}$.
        \item The restriction of $f$ to the set of places of $k$ is $S$-frobenian. This means that there exists some finite extension $K/k$ with Galois group $\Gamma = \Gal(K/k)$ such that $S$ contains all places that ramify in $K/k$ and a class function $\varphi : \Gamma \rightarrow \CC$ satisfying \[\varphi(\Frob_{v}) = f(\PR_v)\] for all $v \not\in S$.
    \end{itemize}
\end{defi} 

\begin{rem}
    If the frobenian multiplicative function $f$ satisfies that $f(\mathfrak{p}^n) \leq H$ for all primes $\PR$ and $n \in \NN$, the second bullet point in Definition \ref{frob funct} follows from the first.
\end{rem}

Note that a class function is a function that is constant on conjugacy classes. In particular, this means that $\varphi(\Frob_v)$ is well defined for $v \not\in S$, so we adopt an abuse of notation by letting $\Frob_v$ be a choice of element in the Frobenius conjugacy class. The mean, $m(f)$, of a frobenian multiplicative function is defined to be the mean of the associated frobenian function, given by \[m(f) = \frac{1}{|\Gamma|} \sum_{\sigma \in \Gamma}\varphi(\sigma)\] where $\varphi$ is the class function associated to $f$ and $\Gamma$ the Galois group on which $\varphi$ is defined. The mean of a frobenian multiplicative function can equivalently be described as an integral over the absolute Galois group $\Gamma_k = \Gal(\overline{k}/k)$,
\begin{equation}
m(f) = \int_{\gamma \in \Gamma_k}\varphi(\gamma) d\gamma. 
\end{equation}
Working over the absolute Galois group instead of the Galois group associated to a specific frobenian function makes it easier to relate two frobenian functions, and in particular to take the mean of their sum and products. If $\varphi: G \rightarrow \CC$ is the class function corresponding to $f$, then we can view $G$ as the quotient $\Gamma_k/N$ of the absolute Galois group by a normal subgroup $N$, so that $\varphi$ is a $N$-invariant function on $\Gamma_k$. To take the integral we equip $\Gamma_k$ and $N$ with their respective Haar probability measures, which then gives a Haar measure on $G$. We also have the following: \begin{defi}[$L$-Function of a frobenian multiplicative function]
    Let $f$ be a frobenian multiplicative function. Then its $L$-function is defined to be the Dirichlet series \[L(f,s) = \sum_{\mathfrak{n}}\frac{f(\mathfrak{n})}{N(\mathfrak{n})^s} = \prod_{\mathfrak{p}} \left( 1+ \sum_{j \geq 1} \frac{f(\mathfrak{p}^j)}{N(\mathfrak{p})^{js}}\right) \] for $\re(s) > 1$, and it converges absolutely in this region as $|f(\mathfrak{n})| \leq C_{\epsilon}N(\mathfrak{n})^{\epsilon}$.
\end{defi} We say that a set is \emph{frobenian} if its indicator function is frobenian. A function is \emph{frobenian} if it is $S$-frobenian for some set $S$ of places of $k$. 

\begin{exa}[Examples of frobenian multiplicative functions]

The following are examples of frobenian multiplicative functions when $k = \QQ$:
    \begin{enumerate}
        \item The divisor function, which is the arithmetic function that counts the number of divisors of an integer.
        \item The indicator function for the sums of two squares. 
        \item An example due to Serre \cite[$\S 3.4.3$]{serre2016lectures} is the function given by the mapping $p \mapsto a_p \bmod mA$, where $p$ is a prime and $a_p$ is the $p$th Fourier coefficient of a modular form on $\Gamma_0(N)$ for some fixed level $N$ and weight $k>0$ (in order to make sure this takes values in $\CC$, it is necessary to embed $\ZZ/(mA)\ZZ$ into $\CC^{\times}$ via the mapping $x \mapsto e^{2\pi i x}$). Then this is an $S_{mN}$ frobenian function where $S_{mN}$ contains the primes dividing $mN$.
    \end{enumerate}
\end{exa}

We will need the following result on frobenian functions stated and proved in \cite[Prop. $2.3$]{Frei_2022}. It uses the Dedekind zeta function for a number field $k$, which is the function given by \[\zeta_k(s) = \sum_{\substack{\mathfrak{a} \subset \mathcal{O}_k \\ \mathfrak{a} \neq 0}}\frac{1}{N(\mathfrak{a})^s}.\] Here $N(\mathfrak{a})$ is the norm of $\mathfrak{a}$, which is given by the number of elements of the (finite when $\mathfrak{a} \neq 0$) ring $\mathcal{O}_k/\mathfrak{a}$. If $v$ is a non-archimedean place of $k$ then $\zeta_{k,v}$ is the Euler factor of $\zeta_k$ at $v$ and if $v$ is an archimedean place then $\zeta_{k,v} = 1$. 

\begin{prop} \label{prop 2.3}
    Let $k$ be a number field, $S$ a finite set of places of $k$ containing the archimedean places and $\lambda$ be an $S$-frobenian function. Assume that for all $v \not\in S$, $\lambda$ satisfies that $|\lambda_x(v)| < q_v$. Then the Euler product given by \[F(s) = \prod_{v \not\in S} \left(1+\lambda_x(v)q_v^{-s}\right)\] can be written as \[F(s) = \zeta_k(s)^{m(\lambda)}G(x;s)\] for some $G(x;s)$ holomorphic in the region 
    \begin{equation} \label{reg}
        \re(s) > 1 - \frac{c}{\log(|\im(s)|+3)}
    \end{equation}
    
    for some $c<\frac{1}{4}$, and in this region satisfies the bound 
    \begin{equation} \label{bound}
        G(x;s) \ll (|\im(s)|+1)^{\frac{1}{2}}.
    \end{equation} Furthermore,  \[\lim_{s \rightarrow 1}(s-1)^{m(\lambda)}F(s) = (\Res_{s=1} \zeta_k (s))^{m(\lambda)} \prod_{v \not\in S}\frac{1+\lambda_x(v)q_v^{-1}}{\zeta_{k,v}(1)^{m(\lambda)}} \prod_{v \in S}\frac{1}{\zeta_{k,v}(1)^{m(\lambda)}}.\] 
\end{prop}

\begin{proof}
    See \cite[Prop. $2.3$]{Frei_2022}.
\end{proof}

\section{Counting with frobenian multiplicative functions}

Let $f$ be a frobenian multiplicative function, $G$ a finite abelian group, and $\varphi$ be a $G$-extension of $k$. We would like to find an asymptotic formula for the counting function  \[N(k,G,B,f) = \sum_{ \substack{\Phi(\varphi) \leq B\\ \varphi \in G-\text{ext}(k)}}f(\cc(\varphi)),\]which counts the number of $G$-extensions of $k$ such that the norm of its conductor is bounded and whose conductor satisfies some frobenian condition.

\subsection{Dirichlet series, class field theory}

We consider the following Dirichlet series \[D_f(s) = \sum_{\varphi \in G - \ext(k)} \frac{f(\mathfrak{c}(\varphi))}{\Phi(\varphi)^s}\] where $f$ is a frobenian multiplicative function. This is holomorphic on $\re(s) > 1$ since by the definition of a frobenian multiplicative function, for all $\epsilon>0$ there exists some constant $C_{\epsilon}$ such that $|f(\mathfrak{n})| \leq C_{\epsilon} N(\mathfrak{n})^{\epsilon}$. When $f$ is the constant function $f= 1$ then this follows from \cite[Lem.~$2.10$]{Wood_2010}. To simplify notation, we will from now on write $f(\varphi)$ in place of $f(\mathfrak{c}(\varphi))$, with the understanding that $f$ is defined on the conductor of the number field corresponding to $\varphi$. In order to make the counting easier, we wish to remove the surjectivity condition, so that instead of summing over $G$-extensions of $k$ we sum over continuous homomorphisms instead.

Let $\Gamma_k$ denote the absolute Galois group $\Gal(\overline{k}/k)$. Notice that \begin{align*}
    \# \{ \varphi: \Gamma_k \rightarrow G: \Phi(\varphi) \leq B, \: \varphi \: \text{surj.}\} = & \: \# \{ \varphi: \Gamma_k \rightarrow G: \Phi(\varphi) \leq B\} \\
    & + O \left( \sum_{\substack{H \subset G \\ H \neq G}} \# \{\varphi: \Gamma_k \rightarrow H :\Phi(\varphi) \leq B \} \right).
\end{align*} Using this we will analyse the series \[\sum_{\varphi \in \Hom(\Gamma_k, H)} \frac{f(\varphi)}{\Phi(\varphi)^s}\] for an arbitrary subgroup $H$ of $G$ and apply the asymptotic to all $H \subset G$. We will eventually show that the only subgroup that contributes to the leading term is $G$ itself, as we are ordering by conductor, as done in \cite{Wood_2010}. Let $S$ be a finite set of places that contains those places which ramify in the field extension defining $f$, all places $v$ lying above primes $p \leq |G|$, all the archimedean places of $k$, and such that \[\mathcal{O}_S = \{x \in k : v(x) \geq 0 \: \forall \: v \not\in S\}\] has trivial class group. The global Artin map $\AAA^{\times}/k^{\times} \rightarrow \Gal(k^{\text{ab}}/k)$ where $k^{\text{ab}}/k$ is the maximal abelian extension of $k$ induces a canonical identification between the groups \[\Hom(\Gamma_k,H) = \Hom(\mathbb{A}^{\times}/k^{\times},H).\] Moreover the local Artin map $k_v^{\times} \rightarrow \Gal(k_v^{\text{ab}}/k_v)$ induces a canonical identification between \[\Hom(\Gamma_{k_v},H) = \Hom(k_v^{\times},H).\]In light of the identifications above we can now view $f$ as a function defined on $\Hom(\mathbb{A}^{\times}/k^{\times},H)$ instead. Every sub-$G$-extension $\varphi \in \Hom(\Gal(\overline{k}/k),H)$ then corresponds to a continuous homomorphism $\chi \in \Hom(\mathbb{A}^{\times}/k^{\times},H)$. From now on, by an abuse in notation we shall refer to such a $\chi$ as a character. Furthermore, for such a character $\chi \in \Hom(\AAA^{\times}/k^{\times}, H)$ we have that $\Phi(\chi)$ is the reciprocal of the idelic norm of the conductor of $\ker(\chi)$. This corresponds to $\Phi(\varphi)$ for the associated sub-$G$-extension of $k$. We can now consider the series \[D_f(s) =  \sum_{\chi \in \Hom(\AAA^{\times}/k^{\times},H)} \frac{f(\chi)}{\Phi(\chi)^s}\] defined over the ideles, and $f$ is defined on the conductor of $\ker(\chi)$. This puts us in the correct setting for harmonic analysis.

\subsection{Harmonic analysis}

Since $\AAA^{\times}/k^{\times}$ and $k_v^{\times}$ are locally compact abelian groups and $H$ is finite, the groups $\Hom(\mathbb{A}^{\times}/k^{\times},H)$ and $\Hom(k_v^{\times},H)$ are also both locally compact abelian groups \cite[Page~$377$.~Cor.]{Moskowitz1967HomologicalAI}. By \cite[Lem.~$3.2$]{HNPabelianext} their Pontryagin duals can canonically be identified with $\AAA^{\times}/k^{\times} \otimes H^{\wedge}$ and $k_v^{\times} \otimes H^{\wedge}$ respectively, and they have associated pairings $\langle \cdot,\cdot \rangle: \Hom(\mathbb{A}^{\times}/k^{\times},H) \times (\AAA/k^{\times} \otimes H^{\wedge}) \rightarrow S^{1}$ and $\langle \cdot,\cdot \rangle:\Hom(k_v^{\times},H) \times( k_v^{\times} \otimes H^{\wedge} ) \rightarrow S^{1}$. Any locally compact abelian group comes with a Haar measure, which is unique up to scalar multiplication by a positive factor. In particular $\Hom(k_v^{\times},H)$ can be equipped with a unique Haar $d\chi_v$ measure satisfying \[\text{vol}(\Hom(k_v^{\times}/\mathcal{O}_v^{\times}, H)) = 1. \] We choose our measure to be such that for non-archimedean $v$ it is $|H|^{-1}$ times the counting measure and for archimedean places it is the counting measure (using the convention that for archimedean places $\mathcal{O}_v = k_v$). We can take the product of the local measures to obtain a well-defined global measure $d\chi$ on $\Hom(\AAA^{\times},H)$. For $\chi \in \Hom(\AAA^{\times},H)$ then we define  $\chi_v \in \Hom(k_v^{\times},H)$ to be the restriction of $\chi$ to $k_v^{\times}$ for each place $v$. Furthermore, we have the following definition.

\begin{defi} \label{rami}
    Let $\chi_v \in \Hom(k_v^{\times},H)$ and $\Phi_v(\chi_v)$ be the reciprocal of the $v$-adic norm of the conductor of $\ker(\chi_v)$, where the conductor of $\ker(\chi_v)$ is $\mathfrak{p}_v^n = \pi_v^n\mathcal{O}_v$, where $n$ is the smallest integer such that $1+ \mathfrak{p}_v^n \subseteq \ker(\chi_v)$. \\
     We say $\chi_v$ is \textit{unramified} if it is trivial on $\mathcal{O}_v^{\times}$, that is, $\chi_v \in \Hom(k_v^{\times}/\mathcal{O}_v^{\times},H)$. We say $\chi_v$ is \textit{tamely ramified} at $v$ if $\gcd(q_v,|G|) = 1$. In the case that $\chi_v$ is unramified we have $\Phi_v(\chi_v) = 1$ and when the ramification is tame we have $\Phi_v(\chi_v) = q_v$.
\end{defi} It is clear by our assumptions on $S$ that the ramification is tame at all $v \not\in S$. Let $f_v$ be the local function given by restricting $f$ to $\Hom(k_v^{\times}, H)$ for each place $v$,  defined on the conductor of $\ker(\chi_v)$. Due to the compatibility between local and global class field theory we have the relation \[ f = \prod_{v} f_{v}\] and in particular \[f/\Phi^s = \prod_v f_{v}/\Phi_v^s.\] For $v \not\in S$, $\Phi_v$ takes the value $1$ on unramified elements, and since $f_{v}$ is multiplicative, it is also $1$ at these elements. Therefore $f/\Phi^s$ extends to a well defined, continuous function on $\Hom(\AAA^{\times}, H)$ as $f_v/\Phi_v^s$ takes value $1$ on the unramified elements. Using this and the above measure we can now define the Fourier transform of the global function $f/\Phi^s$ and the local functions $f_{v}/\Phi_v^s$ by \[\hat{f}_{H}(x;s) = \int_{\chi \in \Hom(\AAA^{\times},H)} \frac{f(\chi)\langle \chi, x \rangle}{\Phi(\chi)^s} d\chi\] for $x \in \AAA^{\times} \otimes H^{\wedge}$ and \[\hat{f}_{v,H}(x_v;s) = \int_{\chi_v \in \Hom(k_v^{\times},H)} \frac{f_{v}(\chi_v)\langle \chi_v, x_v \rangle}{\Phi_v(\chi_v)^s} d\chi_v\] for $x_v \in k_v^{\times} \otimes H^{\wedge}$. The global Fourier transform  exists for $\re(s) \gg 1$ and has an Euler product decomposition \begin{equation} \label{euler prod}
    \hat{f}_{ H}(x;s) = \prod_v \hat{f}_{v, H}(x_v;s).\end{equation}

\begin{lem} \label{loc fou prop}
    The local Fourier transforms $\hat{f}_{v,H}(x_v;s)$ satisfy $\hat{f}_{v,H}(x_v;s) \ll_{H} 1$ for $\re(s) \geq 0$ and when $f_v$ is real valued and non-negative, we have $\hat{f}_{v,H}(1;s)> 0 $ for $s \in \RR$.
\end{lem}

\begin{proof} We will prove the statement for the non-archimedean places of $k$, as the proof for the archimedean places is the same but without the factor of $\frac{1}{|H|}$. We have chosen our measure to be such that \[ \hat{f}_{v,H}(x_v;s) = \frac{1}{|H|}\sum_{\chi_v \in \Hom(k_v^{\times}, H)}\frac{f_v(\chi_v)\langle \chi_v,x_v\rangle}{\Phi_v(\chi_v)^s}.\] This sum is finite and since $\Phi_v(\chi_v) \neq 0 $ for all $\chi_v$, the local Fourier transform $\hat{f}_{v,H}(x_v;s)$ is holomorphic on $\CC$ (except for possible branch point singularities).

As $f_v$ is frobenian multiplicative, by definition it satisfies that there exists some $H \in \NN$ such that $|f_v(\mathfrak{p}_v^k)| \leq H^k$ for primes $\mathfrak{p}_v$ of $k_v$, meaning that each summand is bounded. Furthermore, the number of summands is bounded by $\ll_{k,H} 1$ giving the second result. For the final part, let $f_v$ be real valued and non-negative, and consider $\hat{f}_{v,H}(x_v;s)$ when $x_v=1$. Here we have

    \begin{align*}
        \hat{f}_{v,H}(1;s) = & \: \frac{1}{|H|}\sum_{\chi_v \in \Hom(k_v^{\times}, H)}\frac{f_v(\chi_v)\langle \chi_v,1\rangle}{\Phi_v(\chi_v)^s} \\
        = & \: \frac{1}{|H|} \sum_{\chi_v \in \Hom(k_v^{\times}, H)} \frac{f_v(\chi_v)}{\Phi_v(\chi_v)^s}.
    \end{align*} By definition  $f$ is a multiplicative function and so takes value $1$ on the trivial homomorphism $\chi_v = \mathbbm{1} \in \Hom(k_v^{\times},H)$. Furthermore $\Hom(k_v^{\times},H)$ is non-empty, so the sum is non-empty, and hence $\hat{f}_{v,H}(1;s)> 0$ for $s \in \RR$.\end{proof}

\subsection{Poisson Summation}

In this section, we will introduce a Poisson summation formula, the proof of which is given in \cite[Prop.~$3.8$]{HNPabelianext} and \cite[Prop.~$3.9$]{Frei_2022}.

\begin{lem}
    The local function $f_v$ is $\Hom(k_v^{\times}/\mathcal{O}_v^{\times},H)$ invariant and $f_v(\psi_v) = 1$ for all $\psi_v \in \Hom(k_v^{\times}/\mathcal{O}_v^{\times},H)$.
\end{lem}

\begin{proof}
    Let $\psi_v \in \Hom(k_v^{\times}/\mathcal{O}_v^{\times},H)$ and $\chi_v \in \Hom(k_v^{\times},H)$. For an unramified homomorphism $\psi_v$, we have $\Phi_v(\chi_v\psi_v) = \Phi_v(\chi_v)$. Moreover, if the conductor $\mathfrak{c}(\chi_v)$ is $\PR_v^n$, we have \[f_v(\cc(\chi_v)) = f_v(\PR_v^{n}) =  f_v(\PR_v^{n} \cdot 1) = f_v(\cc(\chi_v) \cc(\psi_v)) = f_v(\cc(\chi_v\psi_v)). \]
    This means that $f_v$ satisfies $f_v(\chi_v\psi_v) = f_v(\chi_v)$, so $f_v$ is $\Hom(k_v^{\times}/\mathcal{O}_v^{\times},H)$-invariant. The second part follows from setting $\chi_v = \mathbbm{1}$.
\end{proof} It then follows by \cite[Lem.~$3.8$]{Frei_2022} that for $v \not\in S$ we can write \begin{equation} \label{rewrite FT}
    \hat{f}_{v, H}(x_v;s) = \begin{cases}
    \sum_{\chi_v \in \Hom(\mathcal{O}_v^{\times},H)} \frac{f_{v}(\chi_v)\langle \chi_v, x_v \rangle}{\Phi_v(\chi_v)^s}, & \text{if} \: x\in \mathcal{O}_v^{\times} \otimes H^{\wedge}, \\
    0, & \text{otherwise}.
\end{cases}\end{equation} This is due to the fact that the split exact sequence 

 \begin{adjustbox}
{center}\begin{tikzcd}
1 \arrow[r] & \mathcal{O}_v^{\times} \arrow[r] & k_v^{\times} \arrow[r] & k_v^{\times}/\mathcal{O}_v^{\times} \arrow[r] & 1
\end{tikzcd}
\end{adjustbox}
induces the split exact sequence 

 \begin{adjustbox}
{center}\begin{tikzcd}
1 \arrow[r] & \Hom(k_v^{\times}/\mathcal{O}_v^{\times}, H) \arrow[r] & \Hom(k_v^{\times},H) \arrow[r] & \Hom(\mathcal{O}_v^{\times}, H) \arrow[r] & 1.
\end{tikzcd}
\end{adjustbox}
This allows us to split the sum over $\Hom(k_v^{\times},H)$  into two separate sums over $\Hom(k_v^{\times}/\mathcal{O}_v^{\times}, H)$ and $\Hom(\mathcal{O}_v^{\times}, H)$. The result follows from the $\Hom(k_v^{\times}/\mathcal{O}_v^{\times}, H)$-invariance of $f_v$.

In order to relate the Fourier transforms to the original frobenian multiplicative functions we will use the following Poisson summation formula.

\begin{lem}[Poisson summation]
    Let $x \in \mathcal{O}_S^{\times} \otimes H^{\wedge}$. Then for $\re(s)>1$, the global Fourier transform $\hat{f}_{H}(x;s)$ exists and the following formula holds: \begin{equation} \label{poi sum}
        \sum_{\chi \in \Hom(\AAA^{\times}/ k^{\times},H)}\frac{f(\chi)}{\Phi(\chi)^s} = \frac{1}{|\mathcal{O}_k^{\times} \otimes H^{\wedge}|}\sum_{x \in \mathcal{O}_S^{\times} \otimes H^{\wedge} }\hat{f}_{H}(x;s).
    \end{equation} 
\end{lem}

\begin{proof}
    See \cite[Prop.~$3.9$]{Frei_2022} \end{proof} Note that the sum on the right converges as $\mathcal{O}_S^{\times} \otimes H^{\wedge}$ is finite by Dirichlet's $S$-unit theorem. 
\subsection{Local Fourier Transforms}

From now on, we will write $x_v = 1$ to mean that $x_v$ is trivial in $k_v^{\times} \otimes H^{\wedge}$.
\begin{lem}[Local Fourier transforms] \label{loc fou}
    Let $v \not\in S$ and $x_v \in \mathcal{O}_v^{\times} \otimes G^{\wedge}$. Then we have \[\hat{f}_{v, H}(x_v;s) = \begin{cases}
    1 + (|\Hom(\FF_v^{\times}, H)|-1)f(\mathfrak{p}_v)q_v^{-s}, & \text{if} \: x_ v= 1,\\
    1 - f(\mathfrak{p}_v)q_v^{-s}, & \text{if} \: x_v \: \text{is non trivial}.
\end{cases}\] 
\end{lem}

\begin{proof}

We know that for $v \not\in S$, if $\chi_v$ is ramified it is tamely ramified. Clearly for unramified $\chi_v$ we have $\Phi_v(\chi_v) =1$, and for ramified $\chi_v$, we have $\Phi_v(\chi_v) = q_v$. Furthermore, $f_v$ is defined on the conductor of $\ker(\chi_v)$, which is the maximal ideal $\mathfrak{p}_v$ of $\mathcal{O}_v$. By (\ref{rewrite FT}) we have that \[\hat{f}_{v,H}(x_v;s) = \sum_{\chi_v \in \Hom(\mathcal{O}_v^{\times},H)} \frac{f_{v}(\chi_v)\langle \chi_v, x_v \rangle}{\Phi_v(\chi_v)^s}\] and since the unramified homomorphism is trivial on $\mathcal{O}_v^{\times}$ it must contribute $1$ to the sum. Then we can write
\begin{align*}
    \hat{f}_{v,H}(x_v;s) 
& = 1 + q_v^{-s}\sum_{\substack{\chi_v \in \Hom(\mathcal{O}_v^{\times},H) \\ \chi_v \neq \mathbbm{1}}}f_v(\chi_v)\langle \chi_v, x_v \rangle \\
& = 1 + f(\mathfrak{p}_v) q_v^{-s}\sum_{\substack{\chi_v \in \Hom(\mathcal{O}_v^{\times},H) \\ \chi_v \neq \mathbbm{1}}}\langle \chi_v, x_v \rangle. 
\end{align*} We now show that there is an isomorphism between the groups $\Hom(\mathcal{O}_v^{\times},H)$ and $\Hom(\FF_v^{\times},H)$, which the implies that $|\Hom(\mathcal{O}_v^{\times},H)| = |\Hom(\FF_v^{\times},H)|$. We use the argument from \cite[Lem.~$3.10$]{Frei_2022} which we include for completeness. By Hensel's lemma there is a splitting of the exact sequence

 \begin{adjustbox}
{center}\begin{tikzcd}
1 \arrow[r] & 1 + \mathfrak{p}_v \arrow[r] & \mathcal{O}_v^{\times} \arrow[r] & \FF_v^{\times} \arrow[r] & 1
\end{tikzcd}
\end{adjustbox}
and applying the $\Hom( \cdot,H)$ functor we get the split short exact sequence

\begin{adjustbox}
{center}\begin{tikzcd}
1 \arrow[r] & \Hom( \FF_v^{\times} ,H) \arrow[r] & \Hom( \mathcal{O}_v^{\times} ,H) \arrow[r] & \Hom( 1 + \mathfrak{p}_v ,H) \arrow[r] & 1.
\end{tikzcd}
\end{adjustbox}
Thus to show the isomorphism it is enough to show that the set $\Hom( 1 + \mathfrak{p}_v ,H)$ is trivial by showing that there are no non-trivial homomorphisms $1 + \mathfrak{p}_v \rightarrow H$. Consider such a continuous homomorphism. Its kernel contains $1 + \mathfrak{p}_v^n$ for some positive integer $n$. The quotient of any two successive terms in the filtration $1 + \mathfrak{p}_v \supset 1 + \mathfrak{p}_v^2 \supset ... \supset 1 + \mathfrak{p}_v^n$ has order $q_v$ by \cite[$IV$. Prop.~$2.6$]{serre1979local}. Consequently the size of the quotient $1 + \mathfrak{p}_v/1 + \mathfrak{p}_v^n$ must be equal to $q_v^k$ for a positive integer $k$, and this is the size of the image of the homomorphism. However, since $v \not\in S$, the exponent of $H$ must be coprime to $q_v$. It then follows that the image of the homomorphism, and thus the homomorphism itself, is trivial. We have proved that $\Hom( 1 + \mathfrak{p}_v ,H)$ is trivial and by exactness of the above sequence we have that $\Hom( \mathcal{O}_v^{\times} ,H) \cong \Hom( \FF_v^{\times} ,H) $. In particular, they are both the same size. Therefore by character orthogonality \[\sum_{\chi_v \in \Hom(\mathcal{O}_v^{\times},H)} \langle \chi_v, x_v \rangle = \begin{cases}
    |\Hom( \FF_v^{\times} ,H)|, & \: \text{if} \: x_v = 1,\\
    0, & \: \text{otherwise},
\end{cases}\] and when $f(\mathfrak{p}_v) \neq 0$
\[\hat{f}_{v, H}(x_v;s) = \begin{cases}
    1 + (|\Hom(\FF_v^{\times}, H)|-1)f(\mathfrak{p}_v)q_v^{-s}, & \text{if} \: x_v= 1,\\
    1 - f(\mathfrak{p}_v)q_v^{-s}, & \text{if} \: x_v \: \text{is non trivial}.
\end{cases}\] Clearly for $f(\mathfrak{p}_v) = 0$ we have $\hat{f}_{v, H}(x_v;s) = 1$. \end{proof}

\subsection{Exponent} \label{exp section}

We wish to find an expression for the exponent $\varpi(k,G,f)$ in the asymptotic of Theorem \ref{short main thm}. To do this we will use the theory of frobenian functions discussed in $\S 2$. By definition of frobenian multiplicative functions and the assumptions we imposed on $S$ in $\S 3.1$, the function $f$ is an $S$-frobenian function when restricted to the places of $k$. In particular there exists some field extension $K/k$ with Galois group $\Gamma = \Gal(K/k)$ such that $S$ contains all the primes that ramify in $K/k$ and a class function $\varphi: \Gamma \rightarrow \CC$ such that $f(\mathfrak{p}_v) = \varphi(\text{Frob}_v)$ for all $v \not\in S$. Also by our assumptions on $S$, we know from \cite[Cor.~$3.13$]{Frei_2022} that for $x \in \mathcal{O}_{S}^{\times} \otimes H^{\wedge}$ the following function is frobenian for the set $S$:
\[\rho_x(v) = \begin{cases}
    |\Hom(\FF_v^{\times}, H)| - 1, & \text{if} \: x_v = 1,\\
    -1, & \text{if} \: x_v  \: \text{is non trivial}.
\end{cases}\] We define the class function of $\rho$ as follows, following the argument in \cite[$\S 3$]{Frei_2022}. Let $e$ be the exponent of $H$ and $d_H(v)$ be the function \[d_H(v) = \max\{d:d |\gcd(e,q_v-1)\},\] which is an $S$-frobenian function by \cite[ Lem.~$3.11$]{Frei_2022}, and furthermore satisfies that \[|\Hom(\FF_v^{\times}, H)| = |H[d_H(v)]|.\] Then since $d_H$ is an $S$-frobenian function, we can use its class function to define the class function of $\rho$. Let  \[\Sigma_d = \Gal(k(\zeta_{e})/k(\zeta_d)) \backslash \bigcup_{\substack{d'|\frac{e}{d}\\ d' \neq 1}} \Gal(k(\zeta_{e})/k(\zeta_{dd'})) \subset \Gal(k(\zeta_{e})/k).\] We will set $\psi:\Gal(k(\zeta_{e})/k) \rightarrow \CC $ to be the function that takes the constant value $d$ on $\Sigma_d$ for every $d$ dividing $e$. This function $\psi$ is the class function of $d_H$. As we can write $\rho$ as the function \[\rho_x(v) = \begin{cases}
    |H[d_H(v)]| - 1, & \text{if} \: x_v = 1, \\
    -1, & \text{otherwise},
\end{cases}\] we can define its associated class function, defined on $\Gal(k(\zeta_e)/k)$ as the function sending $\sigma \mapsto |H[\psi(\sigma)]| - 1$. 

\begin{lem} \label{lambda}
    The function $\lambda := \rho_x \cdot f$ given by \[ \lambda_x(v)  = \begin{cases}
    (|\Hom(\FF_v^{\times}, H)| - 1)f(\mathfrak{p}_v), & \text{if} \: x_v = 1, \\
    - f(\mathfrak{p}_v), & \text{if} \: x_v  \: \text{is non trivial},
\end{cases} \] is $S$-frobenian for $x \in \mathcal{O}_{S}^{\times} \otimes H^{\wedge}$. Its class function $\vartheta$ is given by the product of the class functions of $\rho$ and $f$. It is defined on $\Gal(k(\zeta_e)K/k)$, where $k(\zeta_e)K$ is the compositum of the fields $K$ and $k(\zeta_e)$.
\end{lem}
\begin{proof}
The result follows directly from \cite[Lem. $2.3$]{Loughran2019FrobenianMF}, as  $f$ and $\rho_x$ are $S$-frobenian functions. \end{proof} We will write $\varpi(k,H,f,x)$ to be the mean of $\lambda_x$, and define $\varpi(k,H,f) : = \varpi(k,H,f,1)$. 
\begin{prop} \label{inv ld}
   If $f$ is real valued and non-negative, then for all $x \in \mathcal{O}_S^{\times} \otimes H^{\wedge}$, we have $\varpi(k,H,f,x) \leq \varpi(k,H,f)$. Furthermore, when $K$ and $k(\zeta_e)$ are linearly disjoint, we have \[
       \varpi(k,H,f) = m(f)\sum_{g \in H \backslash \{ \text{id}_G \}}\frac{1}{[k(\zeta_{|g|}):k]}.
   \]
\end{prop}

\begin{proof}
    It is clear that $\varpi(k,H,f,x) \leq \varpi(k,H,f,1)= \varpi(k,H,f)$ for all $x \in \mathcal{O}_S^{\times} \otimes H^{\wedge}$. For the second part, let $\Gamma = \Gal(k(\zeta_e)K/k)$, $\Gamma_e = \Gal(k(\zeta_e)/k)$ and $\Gamma_K = \Gal(K/k)$, and recall that for linearly disjoint fields, the Galois group of the compositum of the two fields is the product of the Galois group of each field. We have that $\varpi(k,H,f) = \varpi(k,H,f,1)$ is the mean of the $S$-frobenian function $\lambda'(v) = (|\Hom(\FF_v^{\times}, H)| - 1)f(\mathfrak{p}_v)$, which has class function $\vartheta: \Gal(k(\zeta_e)K/k) \rightarrow \CC$ given by $\vartheta(\sigma) = (|H[\psi(\sigma)]| - 1)\varphi(\sigma)$. So \[ \varpi(k,H,f,1) =  \frac{1}{|\Gamma|}\sum_{\sigma \in \Gamma}\vartheta(\sigma).\] By the linear disjointness of $k(\zeta_e)$ and $K$, there is a canonical isomorphism $\Gamma \cong \Gamma_e \times \Gamma_K$ sending $\sigma \mapsto (\sigma{\big|}_{\Gamma_e}, \sigma{\big|}_{\Gamma_K})$. In particular we have $|\Gamma| = |\Gamma_e| \cdot |\Gamma_K|$. Using this isomorphism we can write 
\begin{align*}
    \varpi(k,H,f,1)  = & \frac{1}{|\Gamma_e|}\sum_{\sigma_1 \in \Gamma_e}(|H[\psi(\sigma_1)]| - 1)\frac{1}{|\Gamma_K|}\sum_{\sigma_2 \in \Gamma_K}\varphi(\sigma_2)\\
    = & \: m(f) \sum_{g \in H \backslash \{ \text{id}_G \}}\frac{1}{[k(\zeta_{|g|}):k]}
\end{align*} where $|g|$ is the order of $g$ in $H$, and this last step follows from the proof of \cite[Lem.~$3.15$]{Frei_2022}.\end{proof} In general, the two fields $K$ and $k(\zeta_e)$ may not be linearly disjoint, and this formula may not hold. In this case $\varpi(k,H,f,1)$ should be calculated using the formula \[\varpi(k,H,f,1) = \int_{\gamma \in \Gal(\overline{k}/k)}\vartheta(\gamma) d\gamma.\]We are now able to state the main analytic result of this paper. \begin{thm} \label{main thm}
    Let $G$ be a non trivial finite abelian group, $S$ a finite set of places of $k$ and $|S_{\fin}|$ be the set of non-archimedean places in $S$. Let $f$ be a real valued and non-negative frobenian multiplicative function such that $f$ is $S$-frobenian when restricted to the places of $k$. Furthermore, for $v \in S$, let $\Lambda_v$ be some non-empty set of sub-$G$-extensions of $k_v$ such that $f$ takes the value $1$ on all elements of $\Lambda_v$. Then we have \[N(k,G,B,f) = c_{k,G,f}B(\log B)^{\varpi(k,G, f)-1}(1+o(1))\] where $\varpi(k,G, f) := \varpi(k,G, f, 1)$ is dependent on $f$. Moreover, when $f$ is real valued and non-negative and $\varpi(k,G, f) = \varpi$ does not take values in the set $\{0,-1,-2,...\}$, the leading constant is given by \begin{align*} c_{k, G,f} = \frac{(\Res_{s=1}\zeta_k(s))^{\varpi}}{\Gamma(\varpi)|\mathcal{O}_k^{\times}\otimes G^{\wedge}|}\sum_{x \in \mathcal{X}(k,G, f)} \prod_{v \not\in S}\zeta_{k,v}(1)^{-\varpi}\sum_{\chi_v \in \Hom(\mathcal{O}_v^{\times},G)} \frac{f_v(\chi_v)}{\Phi_v(\chi_v)} \: \times \\
        \frac{1}{|G|^{|S_{\fin}|}} \sum_{\substack{\chi_v \in \Hom(\prod_{v \in S} k_v^{\times},G) \\ \chi_v \in \Lambda_v}} \frac{1}{\prod_{v \in S}\Phi_v(\chi_v)\zeta_{k,v}(1)^{\varpi}}\sum_{x \in \mathcal{X}(k, G, f)} \prod_{v \in S}\langle \chi_v,x_v \rangle. \end{align*} where \begin{align*}
        \mathcal{X}(k, G, f) = \{x \in \mathcal{O}_S^{\times} \otimes G^{\wedge}: \:  & \text{for all places $v \not\in S$ with}  \: f(\PR_v) \neq 0, \\
        & \text{we have} \: x_v = 1 \in k_v^{\times} \otimes G^{\wedge}\},\end{align*} and this leading constant is non-zero if there exists a sub-$G$-extension $\psi$ that satisfies all the local conditions for places $v \in S$ and $f(\psi) \neq 0$.
\end{thm} It is important to note that the leading constant in Theorem \ref{main thm} is different to the leading constant in Theorem \ref{short main thm}. Here we are counting by $G$-extensions whereas in Theorem \ref{short main thm} we are counting by field extensions. However forgetting the isomorphism only changes the counting results by a factor of $|\text{Aut}(G)|$.

\subsection{Asymptotic for the counting function $N(k,G,B,f)$}

From hereon we will assume that the frobenian multiplicative function $f$ is real valued and non-negative.

\begin{prop}[Analytic properties of the Fourier transform]\label{analytic prop}

Let $x \in \mathcal{O}_S^{\times} \otimes H^{\wedge}$. Then \[\hat{f}_{H}(x;s) = \zeta_k(s)^{\varpi(k,H,f,x)}G(x;s) \] where $G(x;s)$ is holomorphic in the region \begin{equation} \label{hol region}
    \re(s) > 1 - \frac{c}{\log(|\im(s)|+3)}
\end{equation} for some $c<\frac{1}{4}$, and \begin{equation}
    G(x;s) \ll (|\im(s)|+1)^{\frac{1}{2}}. \label{bound}
\end{equation} Furthermore, \[\lim_{s \rightarrow 1}(s-1)^{\varpi(k,H,f,x)}\hat{f}_{H}(x;s) = (\Res_{s=1} \zeta_k (s))^{\varpi(k,H,f,x)} \prod_{v}\frac{\hat{f}_{v, H}(x_v;s)}{\zeta_{k,v}(1)^{\varpi(k,H,f,x)}}.\] 
    
\end{prop}

\begin{proof}
    For $v \not\in S$, the local Fourier transforms $\hat{f}_{v,H}(x_v;s)$ can be written $\hat{f}_{v,H}(x_v;s) = 1 +\lambda_x(v)q_v^{-s}$ where $\lambda$ is the $S$-frobenian function defined in Lemma \ref{lambda}. Consider \[F(s) = \prod_{v \not\in S}\hat{f}_{v,H}(x_v;s).\] As $\lambda$ satisfies the requirements of Proposition \ref{prop 2.3}, we can apply it to $F(s)$. From this application of the proposition we obtain the expression for $F(s)$ in terms of the Dedekind zeta function \[F(s) = \zeta_k(s)^{\varpi(k,H,f,x)}\mathcal{H}(x;s).\] Here $\mathcal{H}(x;s)$ is some function that is holomorphic in the region (\ref{reg}), for some constant $c=c_{\mathcal{H}} <  \frac{1}{4}$ and satisfies the bound (\ref{bound}). By Lemma \ref{loc fou prop}, we know that the Euler factors $\hat{f}_{v, H}(x_v;s)$ for $v \in S$ satisfy $\hat{f}_{v, H}(x_v;s) \ll_H 1$. We can multiply the function $\mathcal{H}(x;s)$ by the local Fourier transforms $\hat{f}_{v, H}(x_v;s)$ for each $v \in S$ to obtain a function $G(x;s)= \prod_{v \in S} \hat{f}_{v, H}(x_v;s) \mathcal{H}(x;s)$ that satisfies the bound (\ref{bound}) and is holomorphic in a region of the form (\ref{reg}) (up to possible branch point singularities) but possibly for a smaller constant $c < \frac{1}{4}$. This gives the expression \[\hat{f}_{H}(x;s) = \zeta(s)^{\varpi(k,H,f,x)}G(x;s), \] proving the first part. Also by Proposition \ref{prop 2.3}, we know that 
    \begin{align*}
        \lim_{s \rightarrow 1}(s-1)^{\varpi(k,H,f,x)}F(s) = (\Res_{s=1} \zeta_k (s))^{\varpi(k,H,f,x)} &\prod_{v \not\in S}\frac{1+\lambda_x(v)q_v^{-1}}{\zeta_{k,v}(1)^{\varpi(k,H,f,x)}} \\
        & \times \prod_{v \in S} \frac{1}{\zeta_{k,v}(1)^{\varpi(k,H,f,x)}},\end{align*} and the explicit expression follows after multiplying both sides by the Euler factors $\hat{f}_{v, H}(x_v;s)$ for $v \in S$. The fact that this limit is non-zero at $x=1$ also follows from Lemma \ref{loc fou prop}, as when $f$ is real valued and non-negative and $s \in \RR$, the local Fourier transforms satisfy $\hat{f}_{v,H}(1;s)>0$. \end{proof} By Lemma \ref{loc fou}, for every subgroup $H$, we can expand each of the $\hat{f}_{H}(x;s)$ into Dirichlet series \begin{equation} \label{dirichlet exp}
     \hat{f}_{H}(x;s) = \sum_{n \geq 1}\frac{a_{n}(H,x)}{n^{s}}. \end{equation} Furthermore, in Proposition \ref{analytic prop} we have shown that $\hat{f}_{H}(x;s)$ has a pole of order $\varpi(k,H,f,x)$ at $s=1$ and a meromorphic continuation to the left of the line $\re(s) = 1$. This means we are now in the position to apply a tauberian theorem due to Selberg and D\'elange in the form of \cite[Thm.~$5.2$]{Tenenbaum1995IntroductionTA} in order to find an asymptotic expression for each Dirichlet coefficient $a_n(H,x)$. We will then use the Poisson summation formula (\ref{poi sum}) to relate this asymptotic back to the Dirichlet coefficients $f_n$, giving the desired result.
\begin{lem}
     Let $x \in \mathcal{O}_S^{\times} \otimes H^{\wedge}$ and $a_n(H,x)$ be the Dirichlet coefficient of the global Fourier transform $\hat{f}_{H}(x;s)$. Then we have
    \[\sum_{n \leq B} a_n(H,x) = c_{k,H,f,x} B(\log B)^{\varpi(k,H,f,x)-1}+ O(B(\log B)^{\varpi(k,H,f,x)-2})\] where \[c_{k,H,f, x} = \frac{1}{\Gamma(\varpi(k,H,f,x))}\lim_{s \rightarrow 1}(s-1)^{\varpi(k,H,f,x)} \hat{f}_{H}(x;s).\]
\end{lem}

\begin{proof}
Write $\varpi= \varpi(k,H,f,x)$.
    We have shown that $\hat{f}_{ H}(x;s) = \zeta_k(s)^{\varpi}G(x;s)$, where $G(x;s)$ is holomorphic in a region of the form (\ref{hol region}) and satisfies the bound (\ref{bound}). We use the argument from \cite[Lem.~$3.17$]{Frei_2022} to rewrite $\hat{f}_H(x;s)$ in the form $\zeta(s)^{\varpi}H(x;s)$ where $H(x;s)$ is holomorphic in the region (\ref{hol region}) and satisfies $H(x;s) \ll (|\im(s)|+3)^{\frac{3}{4}}$ in this region. We may then apply the Selberg-D\'elange method as in \cite[Thm.~$5.2$]{Tenenbaum1995IntroductionTA}. (In particular we set $N=0$ and the $(b_n)$ to be $a_n(H,1)$, which clearly satisfy $|a_n(H,x)| < |a_n(H,1)|$, from the definition of $\hat{f}_{v, H}(x_v;s)$.)
\end{proof} We now need to make sure that the only subgroup of $G$ that contributes to the leading term of the asymptotic is $G$ itself. \begin{lem}
    Let $H$ be a proper subgroup of $G$. Then if $f$ is real valued and non-negative, we have $\varpi(k,H,f,x) < \varpi(k,G,f,x)$.
\end{lem}

\begin{proof}
    For $H \subset G$ a proper subgroup, clearly $|\Hom(\FF_v^{\times}, H)| < |\Hom(\FF_v^{\times}, G)|$. It the follows from the definition of $\lambda_x$.
\end{proof}

\begin{prop}[Asymptotic Formula]\label{asympt form}

Let $ x \in \mathcal{O}_S^{\times} \otimes G^{\wedge}$ and write $\varpi = \varpi(k,G,f)$, where $\varpi(k,G, f)$ is defined as in $\S 3.4$. Then \begin{equation} \label{asymp}
    N(k, G, B, f) = c_{k,G, f} B(\log B)^{\varpi-1}(1+o(1)) \end{equation} where \begin{equation} \label{lead const 1}c_{k,G,f} = \frac{1}{\Gamma(\varpi)|\mathcal{O}_k^{\times}\otimes G^{\wedge}|}\sum_{\substack{\varpi = \varpi(k,G,f,x)\\ x \in \mathcal{O}_S^{\times} \otimes G^{\wedge}}}\lim_{s \rightarrow 1}(s-1)^{\varpi}\hat{f}_{G}(x;s).\end{equation}
\end{prop} 

\begin{proof}
Let $f_n$ be the coefficients of the Dirichlet series $D_f(s)$. By the Poisson summation formula (\ref{poi sum}), for every $H\subseteq G$, the Dirichlet series $D_f(s)$ satisfies \[
        D_f(s) = \frac{1}{|\mathcal{O}_k \otimes H^{\wedge}|} \sum_{x \in \mathcal{O}_S^{\times}\otimes H^{\wedge}}\hat{f}_{H}(x;s),\] and so by the expansion of $\hat{f}_{H}(x;s)$ into the Dirichlet series (\ref{dirichlet exp}), the individual Dirichlet coefficients $f_n$ of $D_f(s)$ each satisfy  \[f_n = \frac{1}{|\mathcal{O}_k^{\times} \otimes H^{\wedge}|}\sum_{x \in \mathcal{O}_S^{\times} \otimes H^{\wedge} } a_n(H,x).\] Then since the only subgroup contributing to the leading term is $G$ itself and $N(k,G,B,f) = \sum_{n \leq B}f_n$, the result follows.
\end{proof}

\subsection{Leading constant of the asymptotic formula}

In this section we will prove the explicit expression for the leading constant $c_{k,G,f}$ of (\ref{asymp}) in Proposition \ref{asympt form}. To do this we want to consider which $x \in \mathcal{O}_S^{\times} \otimes G^{\wedge}$ contribute to the leading singularity in (\ref{lead const 1}). We will assume $m(f) \neq 0$. Let $\mathcal{X}(k, G, f)$ be the set \begin{align*}
    \mathcal{X}(k,G, f) =  \{x \in k^{\times} \otimes G^{\wedge} : \:  & \text{for all but finitely many places with } \: f(\PR_v) \neq 0 , \\
     & \: \text{we have} \: x_v = 1 \in k_v^{\times} \otimes G^{\wedge} \}\end{align*} We wish to write the sum in (\ref{lead const 1}) as a sum over the elements of $\mathcal{X}(k, G, f)$; to do so we must first prove $\mathcal{X}(k, G, f)$ is finite. In the proof of this we will use the following lemma. \begin{lem} \label{gal}
    Let $x \in k^{\times} \otimes G^{\wedge}$, let $q$ be a prime power and $\zeta_q$ be a $q$th root of unity. The group $\Gamma : = \Gal(k(\sqrt[q]{x}, \zeta_q)/k)$ is finite and can be naturally embedded into the semidirect product \[\ZZ/q\ZZ \rtimes (\ZZ/q\ZZ)^{\times}.\] 
\end{lem}

\begin{proof}
    Suppose $\mu_q \subset k$. Then $\Gamma = \Gal(k(\sqrt[q]{x})/k)$ which is equal to some cyclic group of order dividing $q$, which is clearly finite. Also, if $x=1$, then $\Gamma = \Gal(k(\zeta_q)/k)$. This Galois group is canonically isomorphic to $(\ZZ/q\ZZ)^{\times}$, and thus is finite. So suppose instead that $k$ does not contain $\mu_q$ and $x \neq 1$. The intermediate fields of $k(\sqrt[q]{x}, \zeta_q)$ are $K = k(\zeta_q)$ and $L = k(\sqrt[q]{x})$. The automorphism  \[\sigma: \zeta_q \mapsto \zeta_q, \:\ \sqrt[q]{x} \mapsto \zeta_q\sqrt[q]{x}\] of $k(\sqrt[q]{x}, \zeta_q)$ fixes $k(\zeta_q)$, is of order $q$ and in particular generates $\Gal(k(\sqrt[q]{x},\zeta_q)/k(\zeta_q))$. Furthermore, the field extension $k(\zeta_q)/k$ is Galois as $\text{char}(k) = 0$ and $t^q-1$ is separable over $k$. Thus $\Gal(k(\sqrt[q]{x}, \zeta_q)/k(\zeta_q))$ is a normal subgroup of $\Gamma$. Normal subgroups $N \trianglelefteq G$ induce split exact sequences of the form 

\begin{adjustbox}
{center}\begin{tikzcd}
1 \arrow[r] & N \arrow[r] & G \arrow[r] & G/N \arrow[r] & 1.
\end{tikzcd}
\end{adjustbox} 
By letting $G = \Gal(k(\sqrt[q]{x}, \zeta_q)/k)$ and $N = \Gal(k(\sqrt[q]{x}, \zeta_q)/k(\zeta_q))$, we get \[G/N = \Gal(k(\zeta_q)/k) \cong (\ZZ/q\ZZ)^{\times}\] and there is a split exact sequence 

\begin{adjustbox}
{center}\begin{tikzcd}
0 \arrow[r] & \ZZ/q\ZZ \arrow[r] & \Gamma \arrow[r] & (\ZZ/q\ZZ)^{\times} \arrow[r] & 0.
\end{tikzcd}
\end{adjustbox} 
Using the fact that semidirect products arise from split exact sequences, we obtain that $\Gamma \cong \ZZ/q\ZZ \rtimes (\ZZ/q\ZZ)^{\times}$, which is finite and of order $|\ZZ/q\ZZ | \cdot |(\ZZ/q\ZZ)^{\times}|$. 
\end{proof}

\begin{lem} \label{finite}
    The set\begin{align*}
    \mathcal{X}(k,G, f) =  \{x \in k^{\times} \otimes G^{\wedge} : \:  & \text{for all but finitely many places with }  \;  f(\PR_v) \neq 0 , \\
     & \: \text{we have} \: x_v = 1 \in k_v^{\times} \otimes G^{\wedge} \}\end{align*} is finite.
\end{lem}

\begin{proof}

    It is enough to prove this for $G^{\wedge} = \ZZ/q\ZZ$ when $q$ is a power of an odd prime. Let $g_x(v)$ be the indicator function for the condition $x_v = 1 \in k_v^{\times} \otimes G^{\wedge}$. Clearly this is an $S$-frobenian condition so $g_x$ is an $S$-frobenian function. This means that showing that $\mathcal{X}(k,G, f)$ is finite is equivalent to showing that if we have $f(\PR_v)g_x(v) = f(\PR_v)$, then the condition $f(\PR_v) \neq 0 $ implies $g_x(v) = 1$ holds only for finitely many $g_x$. In more generality, this is showing that a given $S$-frobenian function can only correlate with finitely many other $S$-frobenian functions. 

     For $G^{\wedge} = \ZZ/q\ZZ$, the condition $x_v = 1 \in k_v^{\times} \otimes G^{\wedge}$ is the same as $x_v \in k_v^{\times q}$. In particular if $g_x(v) = 1$, then there is a solution in $k_v$ to the polynomial $t^q - x$. The splitting field of this polynomial is $k(\sqrt[q]{x}, \zeta_q)$ where $\zeta_q$ is a primitive $q$th root of unity. Clearly the set $S$ contains all places ramifying in $k(\sqrt[q]{x}, \zeta_q)/k$, and thus we can take the field extension defining $g_x$ to be $ k(\sqrt[q]{x}, \zeta_q)$ and the class function corresponding to $g_x$ to be defined on $\Gamma := \Gal(k(\sqrt[q]{x}, \zeta_q)/k)$. We will denote this class function by $\psi_{g_x}$. By Lemma \ref{gal}, we know that the group $\Gal(k(\sqrt[q]{x}, \zeta_q)/k)$ embeds into a fixed finite group. Furthermore, we know that $f$ has class function $\varphi: \Gal(K/k) \rightarrow \CC$ and that the group $\Gal(K/k)$ is finite. 
By \cite[Lem.~$2.5$]{Loughran2019FrobenianMF} the statement $f(\PR_v)g_x(v) = f(\PR_v)$ for all but finitely many $v$ is equivalent to $m(fg_x) = m(f)$, where $m(fg_x)$ is the mean of $fg_x$. The class function for $fg_x$ is given by the product $\varphi \cdot \psi_{g_x}$ and defined on $\Gamma_{fg_x} := \Gal(k(\sqrt[q]{x}, \zeta_q)K/k)$. The Galois group $\Gamma_{fg_x} $ is isomorphic to a subgroup of $\Gal(k(\sqrt[q]{x}, \zeta_q)/k) \times \Gal(K/k)$, which itself is embedded into a fixed finite group by Lemma \ref{gal}, which we will denote $\Gamma_q$. In particular, the mean \[m(fg_x) = \frac{1}{|\Gamma_q|}\sum_{\gamma \in \Gamma_{q}} \varphi(\gamma)\psi_{g_x}(\gamma)\] is well defined. To show that $f(\PR_v)g_x(v) = f(\PR_v)$ holds for only finitely many $g_x$, we will show that $m(fg_x) = m(f)$ for only finitely many $g_x$. Since $\varphi$ and $\psi_{g_x}$ are class functions defined on finite groups, by character theory they have a decomposition into finitely many irreducible characters, given by \begin{equation}
\varphi = \sum_{\chi} c_{\chi}\chi
   \quad\mathrm{and}\quad 
\psi_{g_x}= \sum _{\eta}d_{\eta}\eta
\end{equation} and furthermore we can define \[ \langle \varphi, \psi_{g_x} \rangle := \frac{1}{|\Gamma_{q}|}\sum_{\gamma \in \Gamma_{q}} \varphi(\gamma)\psi_{g_x}(\gamma)\] using that $\overline{\psi_{g_x}} = \psi_{g_x}$.  After decomposing $\varphi$ and $\psi_{g_x}$ into their irreducible characters (we may need to sum over more characters for uniformising purposes as we are working over a larger Galois group) we obtain \begin{align*}
        \langle \varphi,\psi_{g_x} \rangle  = & \: \langle \sum_{\chi} c_{\chi}\chi , \sum _{\eta}d_{\eta}\eta \rangle \\
        = & \: c_{\chi_1}(d_{\eta_1} \langle \chi_1, \eta_1 \rangle + \cdots + d_{\eta_m} \langle \chi_1, \eta_m \rangle) + \cdots \\
        + & \: c_{\chi_n}(d_{\eta_1} \langle \chi_n, \eta_1 \rangle + \cdots + d_{\eta_m} \langle \chi_n, \eta_m \rangle).
    \end{align*}   By character orthogonality $\langle \chi_i, \eta_j \rangle = 1$ if $\chi_i = \eta_j$ and zero otherwise, which means that $\langle \varphi, \psi_g \rangle \neq 0$ if at least one irreducible character appears in the decomposition of both $\varphi$ and $\psi_{g_x}$. Since there are finitely many irreducible characters in the decomposition of $\varphi$, there can only be finitely many $\psi_{g_x}$ where at least one character $\chi$ of $\sum_{\chi} c_{\chi}\chi$ appears in the decomposition $\sum _{\eta}d_{\eta}\eta$ of $\psi_{g_x}$. Hence there are only finitely many $g_x$ with $m(fg_x) \neq 0$ and so only finitely many $g_x$ such that $f(\PR_v)g_x(v) = f(\PR_v)$ can hold. It follows that the set $\mathcal{X}(k,G,f)$ is finite.
  \end{proof}

\begin{lem}
    The set $\mathcal{X}(k,G,f)$ in Lemma \ref{finite} is equal to \begin{align*}
        \mathcal{X}(k, G, f) = \{x \in \mathcal{O}_S^{\times} \otimes G^{\wedge} : \: & \text{for all places $v \not\in S$ with} \: f(\PR_v) \neq 0, \\
        & \text{we have} \: x_v = 1 \in k_v^{\times} \otimes G^{\wedge} \}.\end{align*}
\end{lem}
\begin{proof}

We are summing over $x \in \mathcal{O}_S^{\times} \otimes G^{\wedge}$ in the Poisson sum, so we only need to consider these $x$. To show that the condition holds for all $v \not\in S$, let $x \in \mathcal{O}_S^{\times} \otimes G^{\wedge}$ and consider the sets $\{v \in \Omega_k : f(\PR_v) \neq 0\}$ and $\{v \in \Omega_k : x_v = 1 \in k_v^{\times} \otimes G^{\wedge}\}$. The first set is clearly $S$-frobenian as $f$ is, and the second is $S$-frobenian by \cite[Lem.~$3.12$]{Frei_2022}. The intersection of two $S$-frobenian sets is $S$-frobenian, and so the result follows. \end{proof}

\begin{lem}
    Let $x \in \mathcal{O}_S^{\times} \otimes G^{\wedge}$. Then $x \in \mathcal{X}(k, G, f)$ if and only if $\varpi(k,G,f,x)=\varpi(k,G,f)$. Furthermore for $x \in \mathcal{X}(k,G,f)$ we have that $\hat{f}_{v, G}(x_v;1)=\hat{f}_{v, G}(1;1)$ for all $v \not\in S$.
\end{lem}

\begin{proof}
Suppose $ x \in \mathcal{X}(k, G, f)$. Then $ x_v = 1 \in k_v^{\times} \otimes G^{\wedge}$ when $f(\PR_v) \neq 0$. By definition, when $x_v = 1$, the function $\lambda$ is given by $\lambda_x(v) = (|\Hom(\FF_v^{\times},G)|-1)f(\mathfrak{p}_v)$ and the mean $\varpi(k,G,f,x)$ of this function is exactly $\varpi(k,G,f,1) = \varpi(k,G,f)$. If $x \not\in \mathcal{X}(k, G, f)$ then there is some $v \not\in S$ such that $f(\PR_v) \neq 0$ but $x_v$ is non-trivial in $k_v^{\times} \otimes G^{\wedge}$. Then \[-f(\mathfrak{p}_v) < (|\Hom(\FF_v^{\times},G)|-1)f(\mathfrak{p}_v) \] since $\Hom(\FF_v^{\times},G)$ always contains the trivial homomorphism. It follows that in this case $\varpi(k,G,f,x) < \varpi(k,G,f, 1)= \varpi(k,G,f) $ and thus the equality only holds for $x \in \mathcal{X}(k,G,f)$. For $x \in \mathcal{X}(k,G, f)$ and $v \not\in S$, we have \[\hat{f}_{v, G}(x_v;1)  = 1+(|\Hom(\FF_v^{\times}, G)|-1)f(\mathfrak{p}_v)q_v^{-1}\] which is exactly $\hat{f}_{v, G}(1;1)$.
\end{proof}

\subsubsection{Calculating the leading constant}\label{calc leading const}

In this section we calculate the explicit formula for the leading constant. From Proposition \ref{asympt form} we know that \[c_{k,G,f} = \frac{1}{\Gamma(\varpi)|\mathcal{O}_k^{\times}\otimes G^{\wedge}|}\sum_{\substack{\varpi = \varpi(k,G,f,x) \\ x \in \mathcal{O}_S^{\times}\otimes G^{\wedge}}}\lim_{s \rightarrow 1}(s-1)^{\varpi}\hat{f}_{G}(x;s).\] We will focus on the inner sum \begin{equation}
         \sum_{\substack{\varpi = \varpi(k,G,f,x)\\ x \in \mathcal{O}_S^{\times}\otimes G^{\wedge}}}\lim_{s \rightarrow 1}(s-1)^{\varpi}\hat{f}_{G}(x;s).
     \end{equation} We know from Lemma $3.14$ that for $x \in \mathcal{O}_S^{\times}\otimes G^{\wedge}$, we have $\varpi = \varpi(k,G,f,x)$ exactly when $x \in \mathcal{X}(k, G, f)$. Furthermore, from Proposition \ref{analytic prop}, \[\lim_{s \rightarrow 1}(s-1)^{\varpi}\hat{f}_{G}(x;s) = (\Res_{s=1} \zeta_k (s))^{\varpi} \prod_{v}\frac{\hat{f}_{v, G}(x_v;1)}{\zeta_{k,v}(1)^{\varpi}}\] so we can write $(3.9)$ as \[ (\Res_{s=1} \zeta_k (s))^{\varpi} \sum_{x \in \mathcal{X}(k,G, f)}\prod_{v}\frac{\hat{f}_{v, G}(x_v;1)}{\zeta_{k,v}(1)^{\varpi}}\]
and \[
\sum_{x \in \mathcal{X}(k,G, f)}\prod_{v}\frac{\hat{f}_{v, G}(x_v;1)}{\zeta_{k,v}(1)^{\varpi}}
 = 
 \sum_{x \in \mathcal{X}(k,G, f)}\left( \prod_{v \not\in S}\frac{\hat{f}_{v, G}(x_v;1)}{\zeta_{k,v}(1)^{\varpi}} \times \prod_{v \in S}\frac{\hat{f}_{v, G}(x_v;1)} {\zeta_{k,v}(1)^{\varpi}}\right) 
        \] We will consider separately the factors corresponding to places $v \not\in S$ and those corresponding to places $v \in S$. By Lemma $3.14$, in the case $v \not\in S$, we have that $\hat{f}_{v, G}(x_v;1) = \hat{f}_{v, G}(1;1)$ for $x \in \mathcal{X}(k, G, f)$, and so
        
        \[
           \hat{f}_{v, G}(x_v;1) = \sum_{\chi_v \in \Hom(\mathcal{O}_v^{\times},G)} \frac{f_{v}(\chi_v)\langle \chi_v, 1 \rangle}{\Phi_v(\chi_v)}  = \sum_{\chi_v \in \Hom(\mathcal{O}_v^{\times},G)} \frac{f_{v}(\chi_v)}{\Phi_v(\chi_v)}.
        \]
On the other hand, for places $v \in S$, we have \begin{align*}
   \prod_{v \in S} \hat{f}_{\Lambda_v}(x_v;1) = & \frac{1}{|G|^{|S_{\fin}|}} \prod_{v \in S} \sum_{\chi_v \in \Hom(k_v^{\times},G)} \frac{f_{v}(\chi_v)\langle \chi_v, x_v \rangle}{\Phi_v(\chi_v)} \\
   = &  \frac{1}{|G|^{|S_{\fin}|}} \sum_{\substack{\chi_v \in \Hom(\prod_{v \in S} k_v^{\times},G) \\ \chi_v \in \Lambda_v}} \frac{1}{\prod_{v \in S}\Phi_v(\chi_v)}\prod_{v \in S}\langle \chi_v, x_v \rangle.
\end{align*}
Then after summing over $\mathcal{X}(k, G, f)$ we obtain the expression 

\begin{align*}
    \sum_{x \in \mathcal{X}(k, G, f)}\prod_{v \in S}\frac{\hat{f}_{v, G}(x_v;1)}{\zeta_{v,k}(1)^{\varpi}} = \frac{1}{|G|^{|S_{\fin}|}}  \sum_{\substack{\chi_v \in \Hom(\prod_{v \in S} k_v^{\times},G) \\ \chi_v \in \Lambda_v}} &  \frac{1}{\prod_{v \in S} \Phi_v(\chi_v) \zeta_{v,k}(1)^{\varpi}} \\
&  \times \sum_{x \in \mathcal{X}(k, G, f)}\prod_{v \in S}\langle \chi_v, x_v \rangle.
\end{align*} This gives the expression in Theorem \ref{main thm}.

\subsection{Positivity of the Leading Constant}

We want to show that for a real valued and non-negative frobenian multiplicative function $f$, the leading constant $c_{k,G,f}$ in Theorem \ref{main thm} is positive if there exists a sub-$G$-extension $\varphi \in \Hom(\Gamma_k, G)$ that satisfies all the local conditions imposed at places $v \in S$ and $f(\varphi) \neq 0$. We will check the terms for $v \not\in S$ and those for $v \in S$ separately. The factors for $v \not\in S$ are \[\sum_{x \in \mathcal{X}(k,G, f)} \prod_{v \not\in S}\zeta_{k,v}(1)^{-\varpi}\sum_{\chi_v \in \Hom(\mathcal{O}_v^{\times},G)} \frac{f_v(\chi_v)}{\Phi_v(\chi_v)}.\] The set $\Hom(\mathcal{O}_v^{\times},G)$ always contains the trivial character $\mathbbm{1}$. By definition $f$ is multiplicative and so $\frac{f_v(\chi_v)}{\Phi_v(\chi_v)} = 1$ for the trivial character. Thus the factors for $v \not\in S$ are positive.

For $v \in S$, it suffices to check the contribution given by $\sum_{x \in \mathcal{X}(k,G, f)} \prod_{v \in S}\langle \chi_v,x_v \rangle$. By character orthogonality we have \[\sum_{x \in \mathcal{X}(k, G, f)} \prod_{v \in S}\langle \chi_v,x_v \rangle = \begin{cases}
    | \mathcal{X}(k,G, f)|, & \text{if} \: \prod_{v\in S}\chi_v \: \text{trivial on} \: \mathcal{X}(k, G, f) \\
    0, & \text{otherwise},
\end{cases}\] so it follows that the sum is positive as long as there is some $\chi \in \Hom(\prod_{v\in S} k_v^{\times}, G)$ such that $\prod_{v\in S}\chi_v$ is trivial on $\mathcal{X}(k,G, f)$. By assumption there is some sub-$G$-extension that realises all the local conditions for $v \in S$, and at this sub-$G$-extension $f$ takes non-zero value. Let $\psi: \AAA^{\times}/k^{\times} \rightarrow G$ be the homomorphism coming from this sub-$G$-extension of $k$. Clearly, for every $x \in k^{\times} \otimes G^{\wedge}$ the relation $\prod_{v \in S} \langle \psi_v, x_v \rangle = 1$ holds and in particular, \[\prod_{v \in S}\langle \psi_v, x_v \rangle = \prod_{v \not\in S} \frac{1}{\langle \psi_v, x_v \rangle}. \] Therefore we have $\prod_{v \in S}\langle \psi_v, x_v \rangle = 1$ if and only if $\prod_{v \not\in S}\langle \psi_v, x_v \rangle = 1$, and it is enough to show  $\langle \psi_v, x_v \rangle = 1$ for every $v \not\in S$ and $ x \in \mathcal{X}(k, G, f)$. But for $v \not\in S$, we have that $x \in \mathcal{X}(k, G, f)$ means that $x_v = 1 \in k_v^{\times}\otimes G^{\wedge}$ for every $v$ such that $f(\mathfrak{p}_v) \neq 0$. In particular for the sub-$G$-extension satisfying all the local conditions we have $x_v = 1 \in k_v^{\times}$. Thus $\langle \psi_v, x_v \rangle = 1$ for every $v \not\in S$ and $ x \in \mathcal{X}(k, G, f)$ and the result follows. This completes the proof of Theorem $3.7$. \hfill \qedsymbol

\section{Proof of results}

\subsection{Proof of Theorem $1.1$}

The proof of Theorem $1.1$ follows from setting the frobenian multiplicative function $f$ to be the indicator function $f_{\square}$ of the sums of two squares. Consider the counting function \[N(k,G,f_{\square},B) = \sum_{\substack{\varphi \in G\ext(k) \\ \Phi(\varphi) \leq B}} f_{\square}(\varphi)\] that counts $G$-extensions whose conductor is of bounded norm and is the sum of two squares. 
By applying Theorem \ref{main thm} to $f_{\square}$ we obtain an asymptotic for the function $N(k,G,B,f_{\square})$ defined above of the form \[N(k,G,B,f_{\square}) = c_{k,G,f_{\square}}B(\log B)^{\varpi(k,G,f_{\square})-1}(1+o(1)),\] where $\varpi(k,G,f_{\square})$ is as in Theorem $1.1$ and $c_{k,G,f_{\square}}$ a non-negative constant. We emphasise that once again, this counting function and the leading constant differ by a factor of $|\text{Aut}(G)|$ to those in Theorem $1.1$ as we are counting by $G$-extensions. To find an explicit expression for the invariant it is necessary to first determine the $S$-frobenian function $\lambda_x : = \rho_x \cdot f$ from $\S 3.5$ in the case that $f$ is the indicator function of the sums of two squares. We obtain the function given by \[\lambda_x(v) = \begin{cases}
    |\text{Hom}(\FF_v^{\times}, G)| - 1, & \text{if} \: x_v =1 \: \text{and} \: q_v \equiv 1 \bmod 4, \\
    -1, & \text{if} \: x_v \: \text{is non-trivial and} \: q_v \equiv 1 \bmod 4, \\
    0, & \text{if} \: q_v \equiv 3 \bmod 4,
\end{cases}\]and we denote the mean of this function by $\varpi(k,G,f_{\square}, x)$. Furthermore,  the exponent $\varpi(k,G,f_{\square})$ which we defined in $\S 3.5$ to be $\varpi(k,G,f_{\square},1)$ is the mean of the $S$-frobenian function given by \[\lambda'(v) = \begin{cases}
    |\text{Hom}(\FF_v^{\times}, G)| - 1, & \text{if} \: q_v \equiv 1 \bmod 4, \\
    0, & \text{if} \: q_v \equiv 3 \bmod 4.
    \end{cases}\]
\begin{prop}
    Let $e$ be the exponent of $G$. Then the invariant $\varpi(k,G,f_{\square})$ is given as follows:
\[\varpi(k,G,f_{\square}) = \frac{1}{2}\sum_{\substack{g \in G \backslash\{\id_G\} \\ \mu_4 \not\subset k(\zeta_{|g|})}}\frac{1}{[k(\zeta_{|g|}):k]} + \sum_{\substack{g \in G \backslash\{\id_G\} \\ \mu_4 \subset k(\zeta_{|g|})}}\frac{1}{[k(\zeta_{|g|}):k]}.\] 
\end{prop}

\begin{proof}

The class function for $\lambda'$ is the function given by \[\vartheta(\sigma) = (|G[\psi(\sigma)]|-1)\varphi(\sigma)\] defined on $\Gal(k(\zeta_e)k(\zeta_4)/k)$. Here $\varphi$ is the class function associated to $f_{\square}$, and satisfies \[\varphi ( \Frob_v) = \begin{cases}
    1, & q_v \equiv 1 \bmod 4, \\
    0, & q_v \equiv 3 \bmod 4.
\end{cases}\] Equivalently, it is the function defined on $\Gal(k(\zeta_4)/k)$ such that \[\varphi(\sigma) = \begin{cases}
    1, & \sigma = \id \\
    0, & \sigma = c,
\end{cases}\] where $c$ is complex conjugation. In particular, it indicates whether or not $\sigma$ fixes $\mu_4$. Let $\ell = \text{lcm}(e,4)$ and $\Gamma = \Gal(k(\zeta_{\ell})/k)$. Then \[
    m(\lambda') =  \frac{1}{|\Gamma|}\sum_{\sigma \in \Gamma}\vartheta(\sigma) = \frac{1}{|\Gamma|}\sum_{\sigma \in \Gamma}(|G[\psi(\sigma)]|-1)\varphi(\sigma).
\] Now $\varphi(\sigma)$ is non-zero only if $\sigma$ fixes $\mu_4$, that is, if $\sigma \in \Gal(k(\zeta_{\ell})/k(\zeta_4))$ so we can write the above sum as \[m(\lambda') = \frac{1}{[k(\zeta_{\ell}):k]}\sum_{\sigma \in \Gal(k(\zeta_{\ell})/k(\zeta_4))}(|G[\psi(\sigma)]|-1).\] We will write $k_4 = k(\zeta_4)$. Using the definition of the class function $\psi$ and recalling that the set $\Sigma_d$ is given by $\Sigma_d = \Gal(k(\zeta_{e})/k(\zeta_d)) \backslash \bigcup_{\substack{d'|\frac{e}{d}\\ d' \neq 1}} \Gal(k(\zeta_{e})/k(\zeta_{dd'}))$ we can further rewrite this as \[m(\lambda') =  \frac{1}{[k(\zeta_{\ell}):k]} \sum_{d|\ell}(|G[d]|-1)|\Sigma_d|.\] Then it follows from \cite[Lem.~$3.15$]{Frei_2022} that this is equal to \[m(\lambda') =  \sum_{g \in G\backslash\{\id_G\}} \frac{1}{[k_4(\zeta_{|g|}): k]}.\]

If $g$ is such that $\mu_4 \subset k(\zeta_{|g|})$, then $k_4(\zeta_{|g|}) = k(\zeta_4,\zeta_{|g|}) = k(\zeta_{|g|})$. On the other hand if $g$ is such that $\mu_4 \not\subset k(\zeta_{|g|})$, then $k_4(\zeta_{|g|}) = k(\zeta_4,\zeta_{|g|}) = k(\zeta_4)k(\zeta_{|g|})$. We have \begin{align*}
     \sum_{g \in G\backslash\{\id_G\}} \frac{1}{[k_4(\zeta_{|g|}): k]}& = \sum_{\substack{g \in G \backslash\{\id_G\} \\ \mu_4 \not\subset k(\zeta_{|g|})}}\frac{1}{[k_4(\zeta_{|g|}):k]} + \sum_{\substack{g \in G \backslash\{\id_G\} \\ \mu_4 \subset k(\zeta_{|g|})}}\frac{1}{[k_4(\zeta_{|g|}):k]} \\
     =& \sum_{\substack{g \in G \backslash\{\id_G\} \\ \mu_4 \not\subset k(\zeta_{|g|})}}\frac{1}{[k(\zeta_4)(\zeta_{|g|}):k]} + \sum_{\substack{g \in G \backslash\{\id_G\} \\ \mu_4 \subset k(\zeta_{|g|})}}\frac{1}{[k(\zeta_{|g|}):k]}= \varpi(k,G,f_{\square}).
\end{align*}  Thus the calculation of the mean reduces to the calculation in the proof of \cite[Lem.~$3.15$]{Frei_2022} and the result follows. \end{proof}

Finally, to prove the Corollary \ref{cor 1} we will use the explicit form of the leading constant to show that it is positive. This follows

By setting $f = f_{\square}$ in the set $\mathcal{X}(k,G,f)$ from Section $3.7$, we obtain the finite set \begin{align*}
    \mathcal{X}(k,G, f_{\square}) = \{x \in \mathcal{O}_S^{\times} \otimes G^{\wedge} : \: \text{for all places $v$ with } & \;  q_v \equiv 1 \bmod 4, \\
     & x_v = 1 \in k_v^{\times} \otimes G^{\wedge} \: \forall \;v \not\in S\}.\end{align*} and by setting $f = f_{\square}$ in the explicit formula for the leading constant in \ref{main thm}, we obtain \begin{align*} c_{k, G,f_{\square}} = \frac{(\Res_{s=1} \zeta_k (s))^{\varpi}}{\Gamma(\varpi)|\mathcal{O}_k^{\times}\otimes G^{\wedge}|}\sum_{x \in \mathcal{X}(k, G, f_{\square})} \prod_{\substack{v \not\in S \\q_v \equiv 1 \bmod 4}}\zeta_{v,k}(1)^{-\varpi} \sum_{\chi_v \in \Hom(\mathcal{O}_v^{\times},G)} \frac{1}{\Phi_v(\chi_v)} \: \times \\
        \frac{1}{|G|^{|S_{\fin}|}} \sum_{\chi_v \in \Hom(\prod_{v \in S} k_v^{\times},G) } \prod_{v \in S}\frac{f_{\square_v}(\chi_v)}{ \Phi_v(\chi_v)\zeta_{v,k}(1)^{\varpi}}\sum_{x \in \mathcal{X}(k,G, f_{\square})} \prod_{v \in S}\langle \chi_v,x_v \rangle. \end{align*} To prove Corollary \ref{cor 1}, it is enough to show that the leading constant $c_{k, G,f_{\square}}$ is always positive. This is obtained by taking the sub-$G$-extension $\psi$ in Theorem $3.7$ to be the one corresponding to the trivial character, which gives non-vanishing.

\subsection{Completely split primes in finite extensions}

In this subsection we fix a finite field extension $L/k$. Another example of a frobenian multiplicative function is the indicator function $f_{CS}$ of those ideals such that every prime appearing in its prime factorisation is completely split. If $E/k$ is an abelian field extension, the global conductor $\mathfrak{c}(E/k)$ is the product \[\mathfrak{c}(E/k) = \prod_{\mathfrak{p}} \mathfrak{p}^{n}\] where $n=0$ for all but the finitely many ramified primes $\mathfrak{p}$ of $E/k$. Thus $f_{CS}(\mathfrak{c}(E/k)) = 1$ if and only if every prime that ramifies in $E/k$ is completely split in $L/k$. Using this function we can prove the following result on the splitting of primes in finite extensions.
\begin{prop} \label{cor 2}
    Let $k$ be a number field, $G$ a finite abelian group and $L/k$ a finite field extension. Then there exists a finite abelian extension $E/k$ such that every prime that ramifies in $E/k$ splits completely in $L/k$.
\end{prop} As in the previous section, this result is obtained by applying Theorem \ref{main thm} to the frobenian multiplicative function $f_{CS}$ and showing that the leading constant is always positive, which is achieved by using the trivial character.

Furthermore, when restricted to the primes of $L/k$, the function $f_{CS}$ is the indicator function for the set of completely split primes in and thus its mean is given by the density of this set of primes. By the Chebotarev density theorem this density is $\frac{1}{[L:k]}$. So by Proposition \ref{inv ld}, if $L$ is linearly disjoint from the field $k(\zeta_e)$, the invariant $\varpi(k,G,f_{CS})$ is given by \[\varpi(k,G,f_{CS}) = \frac{1}{[L:k]} \sum_{g \in G \backslash \{\text{id}_G \}}\frac{1}{[k(\zeta_{|g|}): k]}.\]

\section{The classifying stack $BG$}

There has been a recent interest in interpreting Malle's conjecture in terms of stacks, and work has been done on this by Ellenberg, Santriano and Zureick-Brown in \cite{ellenbergheights} and by Darda and Yasuda in \cite{darda2024batyrevmaninconjecturedmstacks}. We will now rephrase some of our results in this setting. 

In this section we will prove a result relating Brauer groups of stacks and the condition that the discriminant of a quadratic number field is the sum of two squares. To do so we will first introduce some related terminology. Let $G$ be a finite \'etale group scheme over a field $k$. We will consider the classifying stack $BG$, which we define to be the quotient stack \[BG = \left[ \text{Spec}(k) / G \right]\] which classifies $G$ torsors and their automorphisms. The group scheme $G$ acts on the fibres of the quotient map \[\text{Spec}(k) \rightarrow \left[ \text{Spec}(k) / G \right] \] making it into a $G$-torsor (in fact this is the universal $G$-torsor). The $k$ points of $BG$ form the groupoid $BG(k)$ of right $G$-torsors over $\text{Spec}(k)$, where the isomorphisms are $G$-equivariant morphisms. When $G$ is a constant group scheme, this is equivalent as a category to the groupoid of homomorphisms $\Gamma_k \rightarrow G$, with morphisms given by conjugation in $G$. We are interested in the case where $k = \QQ$ and $G$ is the constant group scheme $\ZZ/2\ZZ$. By the above discussion $B\ZZ/2\ZZ$ has a $\ZZ/2\ZZ$ torsor given by $\pi :\text{Spec}(\QQ) \rightarrow B\ZZ/2\ZZ$. Let $[\pi] \in \ch^{1}(B\ZZ/2\ZZ, \ZZ/2\ZZ)$ be its cohomology class and $-1 \in \ch^{1}(B\ZZ/2\ZZ, \mu_2)$. Consider the cup product $-1 \cup [\pi] \in \ch^{2}(B\ZZ/2\ZZ, \mu_2)$ induced by the pairing $\ZZ/2\ZZ \times \mu_2 \rightarrow \mu_2$, and its image in the Brauer group $\br(B\ZZ/2\ZZ)$. We will use this Brauer group element to prove the following proposition.

\begin{prop}
    Let $K/\QQ$ be a quadratic extension and $\Delta_{K/\QQ}$ be its absolute discriminant. Then $\Delta_{K/\QQ}$ is the sum of two squares if and only if the Brauer group element $(-1, [\pi]) \in \br (B\ZZ/2\ZZ)$ specialises trivially in $\br\QQ$.
\end{prop}

\begin{proof}
Let $K = \QQ(\sqrt{d})$ and $\chi$ be the homomorphism $\chi: \Gal(\overline{\QQ}/\QQ) \rightarrow \ZZ/2\ZZ$. The specialisation of $(-1, [\pi]) $ at $\chi$ is given by $(-1,[\pi])(\chi) = (-1, K/\QQ) = (-1,d)$. Since the discriminant of a quadratic extension differs from $d$ by at most a factor of $4$, which is a square, we have that $(-1,d) = (-1, \Delta_{K/\QQ}) $. So it remains to show that $\Delta_{K/\QQ}$ is the sum of two squares if and only if $(-1, \Delta_{K/\QQ}) = 0 \in \br\QQ$. However $(-1, \Delta_{K/\QQ}) = 0 \in \br\QQ$ if and only if $\Delta_{K/\QQ} \in \text{N}_{\QQ(\sqrt{-1})/\QQ}\QQ(\sqrt{-1})^{\times}$. But this means that $\Delta_{K/\QQ} = x^{2}- (-1)y^{2} = x^{2}+y^{2}$, that is, it is the sum of two squares. \end{proof} In particular, in this special case, Theorem $1.1$ gives an asymptotic for the number of rational points of $B\ZZ/2\ZZ$ for which the Brauer group element $(-1, [\pi])$ specialises to $0$. Problems of this type for $\mathbb{P}^{n}$ have been considered by Serre in \cite{Serre2000SpcialisationD}.

\bibliographystyle{amsplain}
\bibliography{ms}

\end{document}